\documentclass[11pt]{article}

\usepackage{amssymb}
\usepackage[english]{babel}
\usepackage{amsmath}
\usepackage{amssymb}
\usepackage{graphicx}
\usepackage{color}
\usepackage{xcolor}
\usepackage{booktabs}
\usepackage{array}   % for \newcolumntype macro
\usepackage{afterpage}
\usepackage{subfig}
\usepackage[font={small,it}]{caption}
\usepackage{dsfont}
\usepackage{graphicx}
\usepackage{color}
\usepackage{amsmath,amsthm,amsfonts,epsfig,setspace}
\numberwithin{equation}{section}

\newcommand{\ds}{\displaystyle}
\def\nm{\noalign{\medskip}}
\newtheorem{thm}{Theorem}[section]

\newtheorem{lem}{Lemma}[section]

 \def\p{\partial}
\def \Vh0{\stackrel{\circ}{V}_h} \def\to{\rightarrow}
   
  \def\om{\omega}
   
\def\l{\label}  \def\f{\frac}  

\def\l|{\left|}
\def\r|{\right|}

\def\bx{{\mathbf{x}}}
\def\by{{\mathbf{y}}}
\def\bnu{{\mathbf{\nu}}}

\def\rd{{\mathrm{d}}}
\def\re{{\mathrm{e}}}
\def\ri{{\mathrm{i}}}

\def\loc{{\mathrm{loc}}}

\def\cH{{\mathcal{H}}}
\def\cA{{\mathcal{A}}}
\def\cS{{\mathcal{S}}}
\def\cK{{\mathcal{K}}}

\def\Ker{{\mathrm{Ker}}}

\def\capacity{{\mathrm{Cap}}}

\newcommand\1{{\ensuremath {\mathds 1} }}

\newcommand{\R}{\mathbb{R}}

\newcolumntype{L}{>{$}l<{$}}
\newcolumntype{C}{>{$}c<{$}}
\newcolumntype{R}{>{$}r<{$}}
\newcommand{\eqnref}[1]{(\ref {#1})}
\newcommand{\lc}
{\mathrel{\raise2pt\hbox{${\mathop<\limits_{\raise1pt\hbox
{\mbox{$\sim$}}}}$}}}

\newcommand{\gc}
{\mathrel{\raise2pt\hbox{${\mathop>\limits_{\raise1pt\hbox{\mbox{$\sim$}}}}$}}}

\newcommand{\ec}
{\mathrel{\raise2pt\hbox{${\mathop=\limits_{\raise1pt\hbox{\mbox{$\sim$}}}}$}}}
\newcommand{\pd}[2]{\frac{\partial #1}{\partial #2}}

\def\be{\begin{equation}} \def\ee{\end{equation}}

\def\bea{\begin{eqnarray}}  \def\eea{\end{eqnarray}}

\def\beas{\begin{eqnarray*}} \def\eeas{\end{eqnarray*}}

\def\bn{\begin{enumerate}} \def\en{\end{enumerate}}

\def\bd{\begin{description}} \def\ed{\end{description}}

\title{Subwavelength phononic bandgap opening in bubbly media\thanks{\footnotesize Hyundae Lee was supported by NRF-2015R1D1A1A01059357 grant.  Hai Zhang was supported by HK RGC grant ECS 26301016 and 
startup fund R9355 from HKUST.}}

\date{}

\author{
Habib Ammari\thanks{\footnotesize Department of Mathematics, 
ETH Z\"urich, 
R\"amistrasse 101, CH-8092 Z\"urich, Switzerland (habib.ammari@math.ethz.ch, brian.fitzpatrick@sam.math.ethz.ch, sanghyeon.yu@sam.math.ethz.ch).} \and Brian Fitzpatrick\footnotemark[2] 
\and Hyundae Lee\thanks{\footnotesize  Department of Mathematics, Inha University,  253 Yonghyun-dong Nam-gu,  Incheon 402-751,  Korea (hdlee@inha.ac.kr).}  \and Sanghyeon Yu\footnotemark[2]  \and Hai Zhang\thanks{\footnotesize 
Department of Mathematics, 
 HKUST,  Clear Water Bay, Kowloon, Hong Kong (haizhang@ust.hk).}}

\begin{document}
\maketitle

\begin{abstract}
The aim of this paper is to show both analytically and numerically the existence of a subwavelength phononic bandgap in bubble phononic crystals. The key is an original formula for the quasi-periodic Minnaert resonance frequencies of an arbitrarily shaped bubble. The main findings in this paper are illustrated with a variety of numerical experiments. 
\end{abstract}

\medskip

\bigskip

\noindent {\footnotesize Mathematics Subject Classification
(MSC2000): 35R30, 35C20.}

\noindent {\footnotesize Keywords: Minnaert resonance, bubble, photonic band gap, layer potentials, acoustic waves.}

 %\tableofcontents

\section{Introduction} \label{sec-intro}

The past decade has witnessed growing interest in the fabrication of artificially engineered materials
to effectively control mechanical waves such as sound waves.  Phononic crystals which consist of periodic arrangement of components with controlled spatial size and elastic properties are typical examples. When excited by an acoustic or elastic wave, phononic crystals may exhibit band gaps, or ranges of frequencies in which the wave cannot propagate through their bulk and decaying exponentially. The bandgaps in phononic crystals are due to destructive interference mechanisms from Bragg scattering of the waves, and can be quite wide. For bandgaps to occur, the period of the structure (or the lattice constant) must be of the order of the wavelength and the contrast in the material parameters must be large \cite{Ammari2009_book,  arma, soussi, figotin, hempel, Lipton}. This limits the use of phononic crystals in applications targeting low frequencies, because phononic crystals would require impractically large geometries \cite{pnas,nature}. 

Based on the realization that composites with locally resonant microstructures can exhibit effective negative elastic parameters at certain frequency ranges, a class of phononic crystals that exhibits bandgaps with lattice constants two orders of magnitude smaller than the wavelength have been fabricated \cite{phononic1}. By varying the size and geometry of the microstructure, it was experimentally shown in \cite{phononic1} (and analytically verified using a simple model in \cite{phononic2}) that one can tune the frequency ranges over which the effective elastic parameters are negative. More recently, resonance has been shown both experimentally and numerically to be another way to prevent elastic waves from propagating in the material in \cite{pnas}. 

In this paper, to demonstrate the opening of a subwavelength phononic bandgap, we consider a periodic arrangement of bubbles and exploit their Minnaert resonance \cite{Minnaert1933}.  
The Minnaert resonant frequency depends on the bulk modulus of the air,  the density of the water and the shape of the bubble \cite{Minnaert1933, Leroy2002, H3a}.  
In the dilute regime, it has been shown in \cite{Ammari_Hai} that around  the Minnaert resonant frequency, an effective medium theory can be derived. Furthermore, {above} the Minnaert resonant frequency, the real part of the effective {{modulus}} is negative and consequently, the bubbly fluid behaves as a diffusive media for the acoustic waves. 
Meanwhile, below the Minnaert resonant frequency, with an appropriate bubble volume fraction, a high contrast effective medium can be obtained, making the superfocusing of waves achievable \cite{Ammari2015_a}.  {These show that the bubbly fluid functions like an acoustic metamaterial and indicate that a subwavelength bandgap opening occurs at the Minneaert resonant frequency \cite{Leroy2009} .} We remark that such behavior is rather analogous to the coupling of electromagnetic waves with plasmonic nanoparticles, which results in effective negative or high contrast dielectric constants for frequencies near the plasmonic resonance frequencies \cite{matias1,matias2}.

In this paper, we provide a mathematical and numerical framework for analyzing bandgap opening in bubble phononic crystal at low-frequencies. Through the application of layer potential techniques, Floquet theory, and Gohberg-Sigal theory we derive an original formula for the quasi-periodic Minnaert resonance frequencies of an arbitrarily shaped bubble, along with proving the existence of a subwavelength bandgap and estimating its width. Our results are complemented by several numerical examples which serve to validate them in two dimensions. Our results formally explain the experimental observations  reported  in \cite{Leroy2009}. They pave the mathematical foundation for the analysis of complex-bubble-based phononic crystals that could have more than one structural period and bubbles of different sizes and shapes.

The paper is organized as follows. In Section \ref{sec-1}
we formulate the spectral problem for a bubble phononic crystal and introduce some basic results regarding the quasi-periodic Green's function. In Section \ref{sec-2} we derive an asymptotic formula in terms of the contrast between the densities of the air inside the bubbles and the fluid outside the bubbles. We prove the existence of a subwavelength bandgap and estimate its width. We also consider the dilute regime where the volume fraction of the bubbles is small. In Section \ref{sec-3} we perform numerical simulations in two dimensions to illustrate the main findings of this paper.  We make use of the multipole expansion method to compute the subwavelength bandgap. The paper ends with some concluding remarks. 
In Appendix \ref{sec-appendix-1}, we collect some useful asymptotic formulas for layer potentials in three dimensions.  Derivations of the two-dimensional quasi-periodic Minnaert resonances are given in Appendix \ref{sec-appendix-2}. In Appendix \ref{sec:appendix_multipole}, we briefly describe the basic ideas behind the multipole expansion method.

\section{Problem formulation and preliminaries results} \label{sec-1}

We first describe the bubble phononic crystal under consideration. 
Assume that the bubbles occupy  $\cup_{n\in \mathbb{Z}^d} (D+n)$ for a bounded and simply connected domain $D$ with $\p D \in C^{1, s}$ with $0<s<1$. 
We denote by $\rho_b$ and $\kappa_b$ the density and the bulk modulus of the air inside the bubbles, respectively, and by $\rho$ and $\kappa$ the corresponding parameters for the background media. 
To investigate its phononic gap we consider the following $\alpha-$periodic equation in the unit cell $Y=[-1/2,1/2]^d$ for $d=2,3$:
\be \label{eq-scattering}
\left\{
\begin{array} {ll}
&\ds \nabla \cdot \f{1}{\rho} \nabla  u+ \frac{\omega^2}{\kappa} u  = 0 \quad \text{in} \quad Y \backslash D, \\
\nm
&\ds \nabla \cdot \f{1}{\rho_b} \nabla  u+ \frac{\omega^2}{\kappa_b} u  = 0 \quad \text{in} \quad D, \\
\nm
&\ds  u_{+} -u_{-}  =0   \quad \text{on} \quad \partial D, \\
\nm
& \ds  \f{1}{\rho} \f{\p u}{\p \bnu} \bigg|_{+} - \f{1}{\rho_b} \f{\p u}{\p \bnu} \bigg|_{-} =0 \quad \text{on} \quad \partial D,\\
\nm
&  e^{-i \alpha \cdot x} u  \,\,\,  \mbox{is periodic.}
  \end{array}
 \right.
\ee
Here, $\partial/\partial \bnu$ denotes the outward normal derivative and $|_\pm$ denote the limits from outside and inside $D$.  

Let
\begin{equation*} % \label{data1}
v = \sqrt{\frac{\kappa}{\rho}}, \quad v_b = \sqrt{\frac{\kappa_b}{\rho_b}}, \quad k= \frac{\omega}{v} \quad \text{and} \quad k_b= \frac{\omega}{v_b}
\end{equation*}
be respectively the speed of sound outside and inside the bubbles, and the wavenumber outside and inside the bubbles. We also introduce two dimensionless contrast parameters
\begin{equation*} % \label{data2}
\delta = \f{\rho_b}{\rho} \quad \text{and} \quad \tau= \f{k_b}{k}= \f{v}{v_b} =\sqrt{\f{\rho_b \kappa}{\rho \kappa_b}}. 
\end{equation*}

By choosing proper physical units, we may assume that the size of the bubble is of order 1.  We assume  that the wave speeds outside and inside the bubbles are comparable to each other and that there is a large contrast in the bulk modulus, that is, $$\delta \ll 1,\quad \tau= O(1).$$

It is known that \eqnref{eq-scattering} has nontrivial solution for discrete values of $\omega$ such as (see \cite{Ammari2009_book})

$$  0 \le \omega_1^\alpha \le \omega_2^\alpha \le \cdots$$
and we have the following band structure of propagating frequencies for the given periodic structure:
$$ [0, \max_\alpha \omega_1^\alpha] \cup [ \min_\alpha \omega_2^\alpha, \max_\alpha \omega_2^\alpha] \cup  [ \min_\alpha \omega_3^\alpha, \max_\alpha \omega_3^\alpha] \cup \cdots. $$

In this paper we investigate whether there is a possibility of  bandgap opening in this structure. 

To do this, we first collect notations and some results regarding  the Green function and the quasi-periodic Green's function for the Helmholtz equation in three dimensions. We refer to \cite{Ammari2009_book} and the references therein for the details.

We  introduce the single layer potential $\mathcal{S}_{D}^{k} : L^2(\p D) \to H^1(\p D), H^1_\loc (\R^3)$ associated with $D$ and the wavenumber $k$  defined by, $\forall \bx \in \R^3,$
$$
	  \mathcal{S}_{D}^{k} [\psi](\bx) :=  \int_{\p D} G^k(\bx, \by) \psi(\by) \rd \sigma(\by),
$$
where $$G^k(\bx, \by) :=  - \f{\re^{\ri k |\bx-\by|}}{4 \pi |\bx-\by|},$$ is the Green function of the Helmholtz equation in $\R^3$, subject to the Sommerfeld radiation condition. Here, $L^2(\p D)$ is the space of square integrable functions and $H^1(\p D)$ is the standard Sobolev space.

We also define the boundary integral operator $(\mathcal{K}_{D}^{k})^* : L^2(\p D) \to  L^2(\p D)$ by
$$
	(\mathcal{K}_{D}^{k})^* [\psi](\bx)  := \mbox{p.v.} \int_{\p D } \f{\p G_k(\bx, \by)}{ \p \nu(\bx)} \psi(\by) \rd \sigma(\by), \quad \forall \bx \in \p D.
$$
Here p.v. stands for the Cauchy principal value. 
 We use the notation $ \pd{}{\nu} \Big|_{\pm}$ indicating
$$ \pd{u}{\nu}\Big|_{\pm}(\bx)= \lim_{t \to 0^+} \langle \nabla u(\bx\pm t\nu(\bx)),\nu(\bx) \rangle,$$ 
with $\nu$ being the outward unit normal vector   to $\p D$. Then the following jump formula holds:
$$ \pd{}{\nu} \Big|_{\pm} \mathcal{S}_D^{k}[\phi](\bx) = \left( \pm \frac{1}{2} I + (\mathcal{K}_D^{k})^* \right)[\phi](\bx),\quad \mbox{a.e.}~\bx\in \p D.$$

Let $Y$ be the unit cell $[0,1]^3$ in $\mathbb{R}^3$. For $\alpha\in [-\pi,\pi[^3$, the function $G^{\alpha, k}$ is defined to satisfy
$$ (\triangle_\bx + k^2) G^{\alpha, k} (\bx,\by) = \sum_{n\in \mathbb{R}^3} \delta(\bx-\by-n) e^{\ri  n\cdot \alpha},$$
where $\delta$ is the Dirac delta function and  $G^{\alpha, k} $ is $\alpha$-quasi-periodic, {i.e.}, $e^{-\ri \alpha\cdot \bx} G^{\alpha, k}(\bx,\by)$ is periodic in $\bx$ with respect to $Y$. It is known that $G^{\alpha, k} $ can be written as
$$ G^{\alpha, k}(\bx,\by) = \sum_{n\in \mathbb{Z}^3} \frac{e^{\ri (2\pi n + \alpha)\cdot (\bx-\by)}}{k^2- |2\pi n + \alpha|^2},$$
if $k \ne |2\pi n + \alpha|$ for any $n\in \mathbb{Z}^3$. 

Let $D$ be a bounded domain in $\mathbb{R}^3$ with a connected Lipschitz boundary satisfying $\overline{D} \subset Y$. We define a quasi-periodic single layer potential $\mathcal{S}_D^{\alpha,k}$ by
$$\mathcal{S}_D^{\alpha,k}[\phi](\bx) = \int_{\partial D} G^{\alpha,k} (\bx,\by) \phi(\by) d\sigma(\by),\quad \bx\in \mathbb{R}^3.$$
Then $\mathcal{S}^{\alpha,k}[\phi]$ is an $\alpha$-quasi-periodic function satisfying the Helmholtz equation $(\triangle + k^2)u=0$. It satisfies a jump formula:
$$ \pd{}{\nu} \Big|_{\pm} \mathcal{S}_D^{\alpha,k}[\phi](\bx) = \left( \pm \frac{1}{2} I +( \mathcal{K}_D^{-\alpha,k} )^*\right)[\phi](\bx),\quad \mbox{a.e.}~\bx\in \p D,$$
where $(\mathcal{K}_D^{-\alpha,k})^*$ is the operator given by
$$ (\mathcal{K}_D^{-\alpha, k} )^*[\phi](\bx)= \mbox{p.v.} \int_{\p D} \pd{}{\nu(x)} G^{\alpha,k}(\bx,\by) \phi(\by) d\sigma(\by).$$
We remark that it is known that $\mathcal{S}_D^{0},~\mathcal{S}_D^{\alpha,0} : L^2(\p D) \rightarrow H^1(\p D)$ are invertible for $\alpha \ne 0$; see \cite{Ammari2009_book}. 

%ADD $\alpha=0$.

\section{Subwavelength bandgaps} \label{sec-2}

We use layer potentials to represent the solution to the scattering problem (\ref{eq-scattering}).
We look for a solution $u$ of~\eqref{eq-scattering} of the form
\be \label{Helm-solution}
u =
\begin{cases}
\mathcal{S}_{D}^{\alpha,k} [\psi]\quad & \text{in} ~ Y \setminus \bar{D},\\
 \mathcal{S}_{D}^{k_b} [\psi_b]   &\text{in} ~   {D},
\end{cases}
\ee
for some surface potentials $\psi, \psi_b \in  L^2(\p D)$. 
Using the jump relations for the single layer potentials, one can show that~\eqref{eq-scattering} is equivalent to the boundary integral equation
\be  \label{eq-boundary}
\mathcal{A}(\omega, \delta)[\Psi] =0,  
\ee
where
\[
\mathcal{A}(\omega, \delta) = 
 \begin{pmatrix}
  \mathcal{S}_D^{k_b} &  -\mathcal{S}_D^{\alpha,k}  \\
  -\f{1}{2}+ \mathcal{K}_D^{k_b, *}& -\delta( \f{1}{2}+ (\mathcal{K}_D^{ -\alpha,k})^*)
\end{pmatrix}, 
\,\, \Psi= 
\begin{pmatrix}
\psi_b\\
\psi
\end{pmatrix}.
\]

Throughout the paper, we denote by $\mathcal{H} = L^2(\p D) \times L^2(\p D)$ and by $\mathcal{H}_1 = H^1(\p D) \times L^2(\p D)$, 
and use $(\cdot, \cdot)$ for the inner product in $L^2$ spaces and $\| \cdot \|_\cH$ for the norm in $\mathcal{H}$.  It is clear that $\mathcal{A}(\omega, \delta)$ is a bounded linear operator from $\mathcal{H}$ to $\mathcal{H}_1$, \textit{i.e.}
$\mathcal{A}(\omega, \delta) \in \mathcal{B}(\mathcal{H}, \mathcal{H}_1)$. Moreover, we can check that the characteristic values of $\mathcal{A}(\omega,\delta)$ can be written as
$$  0 \le \omega_1^\alpha \le \omega_2^\alpha \le \cdots. $$

 We first look at the limiting case when $\delta =0$. The operator $\mathcal{A}(\omega, \delta)$ is a perturbation of
\be  \label{eq-A_0-3d}
 \mathcal{A}(\omega, 0) = 
 \begin{pmatrix}
  \mathcal{S}_D^{k_b} &  -\mathcal{S}_D^{\alpha,k}  \\
  -\f{1}{2}+ \mathcal{K}_D^{k_b,*}& 0
\end{pmatrix}.
\ee

We see that $\omega_0$ is a characteristic value of $\mathcal{A}(\omega,0)$ if only if $(\omega_0 / v_b)^2$ is a Neumann eigenvalue of $D$ or $(\omega_0/v)^2$ is a Dirichlet eigenvalue of $Y\backslash D$ with $\alpha$-quasiperiodicity on $\partial Y$.  Since 
zero is a Neumann eigenvalue of $D$, 
$\omega=0$ is a characteristic value for the operator-valued analytic function $\mathcal{A}(\omega,0)$. {Besides, note that there is a positive lower bound for other Neumann eigenvalues of $D$ and all the Dirichlet eigenvalues of $Y\backslash D$ with $\alpha$-quasiperiodicity on $\partial Y$,}  we can conclude the following result by the Gohberg-Sigal theory \cite{Ammari2009_book, Gohberg1971}.
\begin{lem}
For any $\delta$ sufficiently small, there exists {one and only one} characteristic value 
$\omega_0= \omega_0(\delta)$ in a neighborhood of the origin in the complex plane to the operator-valued analytic function 
$\mathcal{A}(\omega, \delta)$.
Moreover,  
$\omega_0(0)=0$ and $\omega_0$ depends on $\delta$ continuously.
\end{lem}

\medskip

\subsection{The asymptotic behavior of $\omega_1^\alpha$}
In this section we assume $\alpha \ne 0$. 
We define
\begin{equation} \label{defA0}
\mathcal{A}_0 :=\mathcal{A}(0,0)= 
 \begin{pmatrix}
  \mathcal{S}_D&  -\mathcal{S}_D^{\alpha,0}  \\
  -\f{1}{2}+ \mathcal{K}_D^{*}& 0
\end{pmatrix},
\end{equation} 
Here we set for brevity $\cS_D := \cS_D^{k = 0}$, $\cK_D^* := \cK_D^{k=0,*}$. We denote by $\1_{\p D} \in H^1(\p D)$ the constant function on $\partial D$ with value $1$, and by $\mathcal{A}_0^* : \cH_1 \to \cH$ the adjoint of $\mathcal{A}_0$. We choose an element $\psi_0\in L^2(\p D)$ such that
$$ \left( -\frac12 I +\cK_D^*  \right) \psi_0=0,\quad  \int_{\p D} \psi_0 = 1.$$
We define  the capacity of the set $D$,  $\capacity_D$, by
\begin{equation}\label{capacity} 
\cS_D [\psi_0] = -\capacity_D^{-1} \1_{\p D}.
\end{equation}

Then we can easily check that $\Ker (\mathcal{A}_0)$ and $ \Ker (\mathcal{A}_0^*)$ are spanned respectively by
\[
\Psi_0 = \begin{pmatrix}
    \psi_0\\
  \tilde\psi_0\end{pmatrix}
  \quad \text{and} \quad
  \Phi_0 = \begin{pmatrix}
    0\\
  \1_{\p D} \end{pmatrix},
\]
where  $\tilde \psi_0 =( \mathcal{S}_D^{\alpha,0})^{-1} \mathcal{S}_D[\psi_0]$.

We now perturb  $ \mathcal{A}_0 $ by a rank-1 operator $\mathcal{P}_0$ from $\mathcal{H}$ to $\mathcal{H}_1$
given by 
$
\mathcal{P}_0[\Psi]:= (\Psi, \Psi_0)\Phi_0,
$
and denote it by
$
\tilde{\mathcal{A}_0}= \mathcal{A}_0 + \mathcal{P}_0
$.
\begin{lem} The followings hold:
\begin{enumerate}
\item[(i)] $\tilde{\mathcal{A}_0}[\Psi_0]= \| \Psi_0 \|^2 \Phi_0 $, $\tilde{\mathcal{A}_0}^*[\Phi_0] = \| \Phi_0 \|^2\Psi_0$. 

\item[(ii)] The operator $\tilde{\mathcal{A}_0}$ and its adjoint $\tilde{\mathcal{A}_0}^*$ are invertible in
$\mathcal{B}(\mathcal{H}, \mathcal{H}_1)$ and  $\mathcal{B}(\mathcal{H}_1, \mathcal{H})$, respectively. 

\end{enumerate}
\end{lem}

\begin{proof}
By construction, and the fact that $\mathcal{S}_D$ is bijective from $L^2(\p D)$ to $H^1(\p D)$ \cite{Ammari2007_polarizationBook}, we can show that $\tilde{\mathcal{A}_0}$ (hence $\tilde{\mathcal{A}_0}^*$) is bijective. The fact that $\tilde{\cA_0}[\Psi_0] =  \| \Psi_0 \|^2 \Phi_0$ is direct. Finally, by noticing that $\mathcal{P}_0^*[\theta]= (\theta, \Phi_0)\Psi_0$, it follows that 
$
\tilde{\mathcal{A}_0}^* [\Phi_0] = \mathcal{P}_0^*[\Phi_0] = \| \Phi_0 \|^2  \Psi_0.
$
\end{proof}

 Using the results in Appendix \ref{sec-appendix-1}, we can expand $\mathcal{A}(\omega,\delta)$ as
\begin{equation} \label{expdA}
\mathcal{A}(\omega, \delta):=\mathcal{A}_0 + \mathcal{B}(\omega, \delta)
= \mathcal{A}_0 + \omega \mathcal{A}_{1, 0}+ \omega^2 \mathcal{A}_{2, 0}
+ \omega^3 \mathcal{A}_{3, 0} + \delta \mathcal{A}_{0, 1}+ \delta \omega^2\mathcal{A}_{2, 1} + O(| \omega| ^4 + |\delta \omega^3|)
\end{equation}
where 
\[
\mathcal{A}_{1,0} = \begin{pmatrix}
  v_b^{-1} \mathcal{S}_{D,1} & 0 \\
  0& 0
\end{pmatrix},
\,\, \mathcal{A}_{2,0}= 
\begin{pmatrix}
  v_b^{-2} \mathcal{S}_{D,2} &  -v^{-2}\mathcal{S}_{D,1}^\alpha  \\
  v_b^{-2} \mathcal{K}_{D,2}^* & 0
\end{pmatrix},
\,\, \mathcal{A}_{3,0}= 
\begin{pmatrix}
  v_b^{-3}\mathcal{S}_{D,3} & 0  \\
 v_b^{-3} \mathcal{K}_{D,3}^*& 0
\end{pmatrix},
\]
\[
\mathcal{A}_{0, 1}=
\begin{pmatrix}
0& 0\\
0 &  -(\f{1}{2}+ (\mathcal{K}_{D}^{-\alpha,0})^*)
\end{pmatrix},
\,\,  \mathcal{A}_{2, 1}=
\begin{pmatrix}
0& 0\\
0 &  -v^{-2}(\mathcal{K}^{\alpha}_{D,1})^*
\end{pmatrix}.
\]

Since $\tilde{\mathcal{A}_0}= \mathcal{A}_0 + \mathcal{P}_0$,  the equation~\eqref{eq-boundary} is equivalent to 
$$
(\tilde{\mathcal{A}_0} - \mathcal{P}_0 + \mathcal{B}) [\Psi_0 + \Psi_1]=0,
$$
where 
$$
(\Psi_1, \Psi_0)=0
$$

Observe that the operator $\tilde{\mathcal{A}_0} + \mathcal{B}$ is invertible for sufficiently small $\delta$ and $\omega$. Applying $(\tilde{\mathcal{A}_0} + \mathcal{B})^{-1}$ to both sides of the above equation leads to
\be  \label{eq-resonance-2}
\Psi_1= (\tilde{\mathcal{A}_0} + \mathcal{B})^{-1} \mathcal{P}_0 [\Psi_0] - \Psi_0
= \left\| \Psi_0 \right\|^2 (\tilde{\mathcal{A}_0} + \mathcal{B})^{-1}[\Phi_0] - \Psi_0.
\ee
Using the condition $(\Psi_1, \Psi_0)=0$, we deduce that~\eqref{eq-boundary} has a nontrivial solution if and only if
\be  \label{eq-algebraic}
\widetilde{A}(\omega, \delta) := \left\| \Psi_0 \right\|^2 \left( \left((\tilde{\mathcal{A}_0} + \mathcal{B})^{-1} [\Phi_0], \Psi_0\right) - 1 \right)=0.
\ee
Let us calculate $A(\omega, \delta) := \widetilde{A}(\omega, \delta) \left\| \Phi_0 \right\|^2$.  Using the Neumann series
$$
(\tilde{\mathcal{A}_0} + \mathcal{B})^{-1} 
= \left(1+ \tilde{\mathcal{A}_0}^{-1} \mathcal{B}\right)^{-1} \tilde{\mathcal{A}_0}^{-1}
= \left( 1- \tilde{\mathcal{A}_0}^{-1}\mathcal{B} + \tilde{\mathcal{A}_0}^{-1}\mathcal{B}\tilde{\mathcal{A}_0}^{-1}\mathcal{B}- ...\right)\tilde{\mathcal{A}_0}^{-1}, 
$$
and the fact that $\tilde{\mathcal{A}_0}^{-1}[\Phi_0] = \left\| \Psi_0 \right\|^{-2} \Psi_0$ and $(\tilde{\mathcal{A}_0}^*)^{-1}[\Psi_0] = \left\| \Phi_0 \right\|^{-2} \Phi_0$, we obtain that
\begin{align*}
A(\omega, \delta)
=& -\omega \left( \mathcal{A}_{1,0}[\Psi_0], \Phi_0\right)
-\omega^2 \left( \mathcal{A}_{2,0}[\Psi_0], \Phi_0\right)
-\omega^3 \left( \mathcal{A}_{3,0}[\Psi_0], \Phi_0\right)  -\delta \left( \mathcal{A}_{0,1}[\Psi_0], \Phi_0\right)\\
& + \omega^2 \left( \mathcal{A}_{1,0}\tilde{\mathcal{A}_0}^{-1}\mathcal{A}_{1,0}[\Psi_0], \Phi_0\right) 
+ \omega^3 \left( \mathcal{A}_{1,0}\tilde{\mathcal{A}_0}^{-1}\mathcal{A}_{2,0}[\Psi_0], \Phi_0\right) 
+\omega^3 \left( \mathcal{A}_{2,0}\tilde{\mathcal{A}_0}^{-1}\mathcal{A}_{1,0}[\Psi_0], \Phi_0\right) \\
& + \omega \delta \left( \mathcal{A}_{1,0}\tilde{\mathcal{A}_0}^{-1}\mathcal{A}_{0,1}[\Psi_0], \Phi_0\right)  
 + \omega \delta \left( \mathcal{A}_{0,1}\tilde{\mathcal{A}_0}^{-1}\mathcal{A}_{1,0}[\Psi_0], \Phi_0\right) \\
& + \omega^3 \left( \mathcal{A}_{1,0}\tilde{\mathcal{A}_0}^{-1}\mathcal{A}_{1,0}\tilde{\mathcal{A}_0}^{-1}\mathcal{A}_{1,0}[\Psi_0], \Phi_0\right) + O( | \omega |^4 + | \delta | \, |\omega|^2  + | \delta |^2).
\end{align*}
It is clear that $\mathcal{A}_{1,0}^*[\Phi_0]=0$. Consequently, the expression simplifies into
\begin{equation} \label{eq:simplificationA}
\begin{aligned}
A(\omega, \delta) = & 
-\omega^2 \left( \mathcal{A}_{2,0}[\Psi_0], \Phi_0\right)-\omega^3 \left( \mathcal{A}_{3,0}[\Psi_0], \Phi_0\right) +  \omega^3 \left( \mathcal{A}_{2,0}\tilde{\mathcal{A}_0}^{-1}\mathcal{A}_{1,0}[\Psi_0], \Phi_0\right) 
\\
&    -\delta \left( \mathcal{A}_{0,1}[\Psi_0], \Phi_0\right)  + \omega \delta \left( \mathcal{A}_{0,1}\tilde{\mathcal{A}_0}^{-1}\mathcal{A}_{1,0}[\Psi_0], \Phi_0\right) + O( | \omega |^4 + | \delta | \, |\omega|^2  + | \delta |^2).
\end{aligned}
\end{equation}

We now calculate the five remaining terms.
%$\left( \mathcal{A}_{2,0}[\Psi_0], \Phi_0\right)$,  $\left( \mathcal{A}_{3,0}[\Psi_0], \Phi_0\right)$, $\left( \mathcal{A}_{0,1}[\Psi_0], \Phi_0\right)$,
%$\left( \mathcal{A}_{2,0}\tilde{\mathcal{A}_0}^{-1}\mathcal{A}_{1,0}[\Psi_0], \Phi_0\right)$ and 
%$\left( \mathcal{A}_{0,1}\tilde{\mathcal{A}_0}^{-1}\mathcal{A}_{1,0}[\Psi_0], \Phi_0\right)$. In the process, we use Lemma~\ref{lem-appendix13}.

\medskip

\noindent \textbf{$\bullet$ Calculation of $\left( \mathcal{A}_{2,0}[\Psi_0], \Phi_0\right)$}. Using the first point of Lemma~\ref{lem:magicId_3d}, we get
\begin{align*}
\left( \mathcal{A}_{2,0}[\Psi_0], \Phi_0\right)&=v_b^{-2}
\left( \mathcal{K}_{D, 2}^*[\psi_0], \1_{\p D} \right)=  v_b^{-2}
\left( \psi_0, \mathcal{K}_{D, 2}[\1_{\p D}] \right) \\
	& =  - v_b^{-2} \int_{\p D} \psi_0(\bx) \int_D G_0(\bx - \by) \rd \by \rd \sigma(\bx) = - v_b^{-2}  \int_D \cS_D[\psi_0] (\bx) \rd \bx = \dfrac{| D |}{v_b^2 \capacity_D},
\end{align*}
where we used the fact that $ \cS_D[\psi_0] (\bx) = - \capacity_D^{-1}$ for all $\bx \in D$.

\medskip

\noindent \textbf{$\bullet$ Calculation of $\left( \mathcal{A}_{3,0}[\Psi_0], \Phi_0\right)$}. Similarly, using the second point of Lemma~\ref{lem:magicId_3d}, we get
\beas
\left( \mathcal{A}_{3,0}[\Psi_0], \Phi_0\right)&=& v_b^{-3}
\left( \psi_0, \mathcal{K}_{D, 3}[\1_{\p D}] \right)
= v_b^{-3} \left( \psi_0,  \f{\ri | D |}{4\pi} \1_{\p D} \right) =  \f{\ri | D |}{4\pi v_b^3}.
\eeas

\noindent{ \textbf{$\bullet$ Calculation of $\left( \mathcal{A}_{0,1}[\Psi_0], \Phi_0\right)$}. We directly have
$$
\left( \mathcal{A}_{0,1}[\Psi_0], \Phi_0\right)
= -(\tilde\psi_0,\left(1/2 + \mathcal{K}_D^{-\alpha,0}\right)[\1_{\p D}]) .
$$

\noindent  {\textbf{$\bullet$ Calculation of  $\left( \mathcal{A}_{0,1}\tilde{\mathcal{A}_0}^{-1}\mathcal{A}_{1,0}[\Psi_0], \Phi_0\right) $}. We have
\beas
\mathcal{A}_{1, 0}[\Psi_0]&=&
 \frac{1}{v_b} \begin{pmatrix}
   \mathcal{S}_{D,1}[\psi_0]\\
  0
\end{pmatrix}
=  \frac{1}{v_b} \f{-\ri}{4 \pi}  \begin{pmatrix}
  \1_{\p D} \\
   0
\end{pmatrix},\\
\mathcal{A}_{0, 1}^*[\Phi_0]&=&
\begin{pmatrix}
  0\\
  -\left(\f{1}{2}+ \cK_D^{-\alpha,0} \right)[\1_{\p D}]
\end{pmatrix}
.
\eeas

Let us calculate
$
\tilde{\mathcal{A}}_0^{-1} \begin{pmatrix}
  \1_{\p D}\\
  0
\end{pmatrix}
$.
We look for $( a\psi_0, b\tilde\psi_0) \in \cH$ so that
\[
\begin{pmatrix}
  \1_{\p D}\\
  0
\end{pmatrix} = \left( \mathcal{A}_0 + \mathcal{P}_0 \right) 
\begin{pmatrix}
  a\psi_0\\
  b\tilde\psi_0
\end{pmatrix}
= \begin{pmatrix}
  (a-b)\mathcal{S}_D [\psi_0]\\
  0
  \end{pmatrix}
  + (a\|\psi_0\|^2 +b\|\tilde\psi_0\|^2 )\begin{pmatrix}
  0\\
  \1_{\p D}
\end{pmatrix}.
\]
By solving the above equations directly, we obtain
\begin{equation} \label{eq:A0F1}
\tilde{\mathcal{A}}_0^{-1} \begin{pmatrix}
  \1_{\p D}\\
  0
\end{pmatrix} = \dfrac{\capacity_D}{\|\psi_0\|^2 +\|\tilde\psi_0\|^2}
\begin{pmatrix}
-\|\tilde\psi_0\|^2\psi_0\\
\|\psi_0\|^2\tilde \psi_0
\end{pmatrix}.
\end{equation}
It follows that
\[
\left( \mathcal{A}_{0,1}\tilde{\mathcal{A}_0}^{-1}\mathcal{A}_{1,0}[\Psi_0], \Phi_0\right) =  \f{\ri \capacity_D\|\psi_0\|^2 (\tilde\psi_0,\left(1/2 + \mathcal{K}_D^{-\alpha,0}\right)[\1_{\p D}]) }{4 \pi v_b (\|\psi_0\|^2+\|\tilde\psi_0\|^2)}.
\]

\noindent \textbf{$\bullet$ Calculation of  $\left( \mathcal{A}_{2,0}\tilde{\mathcal{A}_0}^{-1}\mathcal{A}_{1,0}[\Psi_0], \Phi_0\right)$}. Using similar calculations, we obtain
\beas
\left( \mathcal{A}_{2,0}\tilde{\mathcal{A}_0}^{-1}\mathcal{A}_{1,0}[\Psi_0], \Phi_0\right) 
&=& \left( \tilde{\mathcal{A}_0}^{-1}\mathcal{A}_{1,0}[\Psi_0], \mathcal{A}_{2,0}^* [\Phi_0]\right) \\
&=&   \f{\ri \capacity_D\|\tilde\psi_0\|^2}{4 \pi v_b^3(\|\psi_0\|^2 +\|\tilde\psi_0\|^2)}  \, \big(\psi_0, \mathcal{K}_{D, 2}[\1_{\p D}] \big) 
= \f{\ri |D|\|\tilde\psi_0\|^2}{4 \pi v_b^3(\|\psi_0\|^2 +\|\tilde\psi_0\|^2)} 
\eeas

\noindent \textbf{$\bullet$ Conclusion.} Considering the above the results, we can derive from~\eqref{eq:simplificationA} that
\begin{equation} \label{eq:Aomegadelta}
\begin{aligned}
A(\omega, \delta) = & - \omega^2 \dfrac{ | D |}{v_b^2 \capacity_D} - \omega^3 \dfrac{\ri c_1 | D |}{4 \pi v_b^3} + c_2\delta + \omega \delta \dfrac{\ri c_1c_2 \capacity_D}{4 \pi v_b}   + O( | \omega |^4 + | \delta | \, |\omega|^2  + | \delta |^2),
\end{aligned}
\end{equation}
where \begin{equation}
c_1:=\frac{\|\psi_0\|^2}{\|\psi_0\|^2+\|\tilde\psi_0\|^2},\end{equation}
and \begin{equation}  \label{defc2}
c_2:=(\tilde\psi_0,\left(1/2 + \mathcal{K}_D^{-\alpha,0}\right)[\1_{\p D}]).\end{equation}

We now solve $A(\omega, \delta) =0$. 
It is clear that $\delta = O(\omega^2)$ and thus $\omega_0(\delta) = O(\sqrt{\delta})$. 
We write 
$
\omega_0(\delta) = a_1 \delta^{\f{1}{2}} + a_2 \delta + O (\delta^{\f{3}{2}})
$,
and get
\begin{align*}
	&  -  \dfrac{ | D |}{v_b^2 \capacity_D} \left( a_1 \delta^{\f{1}{2}} + a_2 \delta + O (\delta^{\f{3}{2}}) \right)^2 
	 - \dfrac{\ri c_1 | D |}{4 \pi v_b^3}\left( a_1 \delta^{\f{1}{2}} + a_2 \delta + O (\delta^{\f{3}{2}}) \right)^3 \\
	 & \qquad
	 +c_2 \delta + \dfrac{\ri c_1c_2\capacity_D}{4 \pi v_b}  \left( a_1 \delta^{\f{3}{2}} + a_2 \delta^2 + O (\delta^{\f{5}{2}}) \right) + O(\delta^2) = 0.
\end{align*}
From the coefficients of the $\delta$ and $\delta^{\f{3}{2}}$ terms, we obtain 
\[
	- a_1^2  \dfrac{| D |}{v_b^2 \capacity_D} + c_2 = 0
	 \quad \text{and} \quad 
	 2 a_1 a_2 \dfrac{- | D |}{v_b^2 \capacity_D} - a_1^3 \dfrac{\ri c_1 | D |}{4\pi v_b^3}  + a_1 \dfrac{\ri c_1c_2\capacity_D}{4 \pi v_b} = 0
\]
which yields 
\[
	a_1 =  \pm \sqrt{ \dfrac{v_b^2 c_2 \capacity_D}{| D |} }
	 \quad \text{and} \quad 
	 a_2 = 0.
\]

Therefore, we obtain
\begin{thm}\label{approx_thm} For $\alpha \ne 0$ and sufficiently small $\delta$, we have
\begin{align}
\omega_1^\alpha= \omega_M \sqrt{c_2} + O(\delta^{3/2}), \label{o_1_alpha}
\end{align}
where $\omega_M= \sqrt{ \frac{\delta v_b^2 \capacity_D }{ |D|} } $ is the (free space) Minnaert resonant frequency.
\end{thm}

 Let us define the $\alpha$-quasi-periodic capacity by 
 \begin{equation}\label{capacityalpha}  \capacity_{D,\alpha}:= -  ((\mathcal{S}_D^{\alpha,0})^{-1} [\1_{\partial D}], \1_{\partial D}).
 \end{equation}
  Then we have 
\begin{align*} c_2 &= - \frac{1}{\capacity_D} ( \left( 1/2 +( \mathcal{K}_D^{-\alpha,0})^*\right)(\mathcal{S}_D^{\alpha,0})^{-1} [\1_{\partial D}], [\1_{\partial D}])\\
&= - \frac{1}{\capacity_D} ((\mathcal{S}_D^{\alpha,0})^{-1} [\1_{\partial D}], \1_{\partial D}) = \frac{\capacity_{D,\alpha}}{\capacity_D},\end{align*}
and \eqnref{o_1_alpha} is written as
$$ \omega_1^\alpha = \omega_{M,\alpha} + O(\delta^{3/2})$$
with $\omega_{M,\alpha}= \sqrt{ \frac{\delta v_b^2 \capacity_{D,\alpha} }{ |D|} }.$ We can see that $$\omega_{M,\alpha}\rightarrow 0$$ as $\alpha\to 0$   because 
 $ \left( 1/2 +( \mathcal{K}_D^{-\alpha,0})^*\right)(\mathcal{S}_D^{\alpha,0})^{-1} [\1_{\partial D}] \rightarrow 0$  and  so $\capacity_{D,\alpha}\rightarrow 0$ as $\alpha\rightarrow0$.

We define $\omega_1^*:= \max_{\alpha} \omega_{M,\alpha}$. Then we deduce the following regarding a bandgap opening.
\begin{thm}\label{main}
For every $\epsilon>0$, there exists $\delta_0>0$  and {$\tilde \omega>  \omega_1^*+\epsilon$} such that 
\begin{equation}
 [ \omega_1^*+\epsilon, \tilde\omega ] \subset [\max_\alpha \omega_1^\alpha, \min_\alpha \omega_2^\alpha]
 \end{equation} 
 for $\delta<\delta_0$.
\end{thm}
\begin{proof}
 Using  $\omega_1^0=0$ and the continuity of $\omega_1^\alpha$ in $\alpha$ and $\delta$, we get $\alpha_0$ and $\delta_1$ such that $\omega_1^\alpha < \omega_1^*$ for every $| \alpha|<\alpha_0$ and $\delta<\delta_1$. Following the derivation of \eqnref{o_1_alpha},  we can check that it is valid uniformly in $\alpha$ as far as $|\alpha| \ge \alpha_0$. Thus there exists $\delta_0 < \delta_1$ such $\omega_1^\alpha \le \omega_1^* +\epsilon$ for $|\alpha| \ge \alpha_0$.  We have shown that $ \max_\alpha \omega_1^\alpha \le \omega_1^*+\epsilon$ for sufficiently small $\delta$. To have  $\min_\alpha \omega_2^\alpha > \omega_1^* +\epsilon$ for small $\delta$, it is enough to check that $\mathcal{A}(\omega,\delta)$ has no small characteristic value other than $\omega_1^\alpha$. For $\alpha$ away from $0$, we can see that it is true following the proof of Theorem \ref{approx_thm}. If $\alpha=0$, we have
\begin{equation}
\mathcal{A}(\omega,\delta)=\mathcal{A}(\omega,0) + O(\delta),\end{equation}
near $\omega_2^0$ with $\delta=0$. Since  $\omega_2^0\ne 0$, we have $\omega_2^0(\delta) > \omega_1^* + \epsilon$ for sufficiently small $\delta$. Finally, using the continuity of $\omega_2^\alpha$ in $\alpha$, we obtain  $\min_\alpha \omega_2^\alpha > \omega_1^* +\epsilon$ for small $\delta$.  This completes the proof.
\end{proof}

\subsection{Dilute case}
We emphasize that our calculations in the previous part hold even for the dilute case as long as $\delta / \eta^2$ is small where $\eta$ is the diameter of $D$.

We state an asymptotic behavior of $ \capacity_{D,\alpha}$ when $D= \eta B$ for a small $\eta$. {Note that $\capacity_{D}= \capacity_{\eta B} = \eta \capacity_{B}$.}
Fix $c>0$, the following holds. 
\begin{lem} \label{lem-dilute} For $|\alpha|> c >0$, we have
\begin{equation}  
 \capacity_{D,\alpha} =  \capacity_{D} - R_\alpha(0)  \capacity_{D} ^2+ O(\eta^3),
 \end{equation}
 where $R_\alpha(\bx):=G^{\alpha,0} (\bx) - G^0(\bx) $.
\end{lem}
\begin{proof}
Since $R_\alpha(\bx)$ is smooth and $R_\alpha(\bx) = R_\alpha(0) + O(|\bx|)$ as $|\bx| \to 0$, we have
\begin{align*} \mathcal{S}_D^{\alpha,0}[\phi] (\eta\bx)
&= \eta \int_{\p B} G^0( \bx - \by) \tilde \phi(\by) d\sigma(\by) + \eta^2 R_\alpha(0)\int_{\p B} \tilde\phi (\by)d\sigma(\by)+ O\left(\eta^3 \|\tilde\phi\| \right)\\
&=\eta\left(\mathcal{S}_B[\tilde\phi] +  \eta R_\alpha(0)\int_{\p B} \tilde\phi  + O\left(\eta^2 \|\tilde\phi\| \right) \right),\end{align*}
with $\tilde \phi(\bx) := \phi(\eta \bx)$. Then
\begin{equation}
(\mathcal{S}_D^{\alpha,0})^{-1}[\1_{\p D}](\eta \bx)= \eta^{-1} \left((\mathcal{S}_B)^{-1}[\1_{\p B}] - \eta R_\alpha(0)(\mathcal{S}_B)^{-1}\left[\int_{\p B} (\mathcal{S}_B)^{-1}[\1_{\p B}] \right] + O\left(\eta^2 \right) \right),
\end{equation}
and so
\begin{align*}
 \capacity_{D,\alpha} =-\eta^2 \int_{\p B}(\mathcal{S}_D^{\alpha,0})^{-1}[\1_{\p D}](\eta \bx)\, d\sigma(x) &= \eta \left( \capacity_{B} - \eta R_\alpha(0) \capacity_{B}^2 + O\left(\eta^2 \right) \right)\\
 &=  \capacity_{D} - R_\alpha(0)  \capacity_{D} ^2+ O(\eta^3).
\end{align*}
\end{proof}

By this approximation we have $ \omega_{M,\alpha} \approx \omega_M$ for $\alpha$ away from $0$ and so $\omega_1^* \approx \omega_M$. Combined with Theorem \ref{main} this means that there is a band gap opening slightly above the Minneart resonance frequency  for a single bubble. It is coherent with results in \cite{Ammari_Hai} showing an effective medium theory for the bubbly fluid as the number of bubbles tends to infinity. It is shown there that near and above the Minnaert resonant frequency, the obtained effective media can have a negative {bulk modulus}.

\section{Numerical illustrations} \label{sec-3}
Recall the formula for the $\alpha$-quasi-periodic Minneart resonance:
\begin{align*}
\omega_1^\alpha = \omega_M \sqrt{\frac{\capacity_{D,\alpha}}{\capacity_D}} + O(\delta^{3/2}). \label{o_1_alpha}
\end{align*}

We want to compare $\omega_{approx}^\alpha:= \omega_M \sqrt{\frac{\capacity_{D,\alpha}}{\capacity_D}}$ with the true $\alpha$-quasi-periodic resonance $\omega_{exact}^\alpha$, which can be obtained through direct calculation of the minimum characteristic value of the operator $\mathcal{A}(\omega, \delta)$ in (\ref{eq-boundary}) using Muller's method \cite{H3a}.

We set the density and the bulk modulus of the bubbles to be $\rho_b = 1$ and $\kappa_b = 1$, respectively. In order to confirm that the formula becomes accurate in the appropriate regime, which features similar wavenumbers inside and outside the bubbles along with, in particular, a high contrast in the bulk modulii, we take the density and the bulk modulus of the background material to be $\rho = \kappa = \delta^{-1} \in (10, 1000)$. We assume that the bubble represented by $D$ is a disk of radius  $R= 0.0125$. In Figure \ref{fig:alpha-qp-formula} we plot $\omega_{approx}^\alpha$ and $\omega_{exact}^\alpha$ against the contrast $\delta^{-1}$ and it is clear that the formula provides a highly accurate approximation when the contrast is sufficiently large.

\medskip

Next we present numerical examples to illustrate subwavelength bandgap openings. As $D$ is a  disk of radius $R$, we apply the multipole expansion method for computing the band structure (for the details, we refer to Appendix \ref{sec:appendix_multipole}). As described in \cite{leslie}, the quasi-periodic Green's function is unsuitable for bandgap calculations due to empty resonance phenomenon. Therefore, we make use of the multiple expansion method which is efficient in the case of disk-shaped bubbles. 

We first consider the dilute case. We set $R=0.05$, $\rho=\kappa=5000$ and $\rho_b=\kappa_b=1$. In this case, we have $\delta=0.0002$. Figure \ref{fig:bandgap_dilute} shows the computed band structure. The points $\Gamma, X$ and $M$ represent $\alpha=(0,0)$, $\alpha=(\pi,0)$ and $\alpha=(\pi,\pi)$, respectively.  We plot the first two characteristic values $\mathcal{A}(\omega, \delta)$ along the boundary of the triangle $\Gamma X M$. It can be seen that  a subwavelength bandgap in the spectrum of $\mathcal{A}(\omega, \delta)$ does exist.  Moreover, the bandgap between the first two bands is quite large. It is also worth mentioning that, by zooming the subwavelength bandgap (on the right in Figure \ref{fig:bandgap_dilute}), one can see that $\omega_1^\ast$ is attained at the point $M$ (that is, $\alpha=(\pi,\pi)$). We used $N=7$ for the truncation order of cylindrical waves. Further numerical experiments indicate that this phenomenon is independent of the bubble radius or position.   

\begin{figure} [!h]
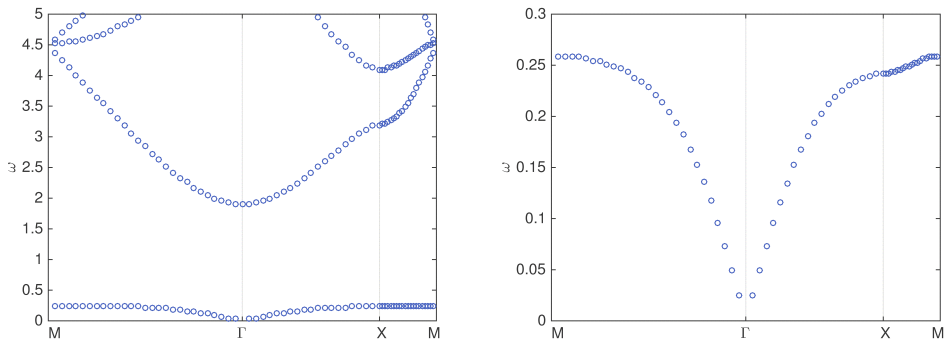

\begin{center}
 \includegraphics[height=10.0cm]{BubbleBand_omega_0to5_r005_c5000}
 \hskip-.5cm
  \includegraphics[height=10.0cm]{BubbleBand_omega_0to03_r005_c5000}
   \caption{ (Dilute case)
  The band structure of a square array of circular bubbles with radius $R=0.05$ and contrast $\delta^{-1} = 5000$.}
  \label{fig:bandgap_dilute}
\end{center}
\end{figure}

Next we consider a non-dilute regime. We set $R=0.25$ and $\rho=\kappa=1000$ and $\rho_b=\kappa_b=1$. In this case we have $\delta=0.001$.
Figure \ref{fig:bandgap_nondilute} shows  the computed band structure. 
%The points $\Gamma, X$ and $M$ represent $\alpha=(0,0)$, $\alpha=(\pi,0)$ and $\alpha=(\pi,\pi)$, respectively.  We plot the characteristic values $\omega$ along the boundary of the triangle $\Gamma X M$. 
Again, a subwavelength bandgap can be  observed.
%Moreover, the bandgap between the first and second bands is shown. 
%For convenience, we also plot the magnified subwavelength band by changing the range of $\omega$ to a low frequency region $\omega\in (0,0.3)$(the right figure of Figure \ref{fig:bandgap_dilute}). It is worth to mention that, in the first band, the maximum of $\omega$ is attained at the point $M$ (that is, $\alpha=(\pi,\pi)$).
We used $N=3$ for the truncation order of cylindrical waves in the  multipole expansion method.

\begin{figure} [!h]
\begin{center}
 \includegraphics[height=10.0cm]{BubbleBand_omega_0to5_r025_c1000}
 \hskip-.5cm
  \includegraphics[height=10.0cm]{BubbleBand_omega_0to03_r025_c1000}
   \caption{ (Non-dilute case)
  The band structure of a square array of circular bubbles with radius $R=0.25$ and contrast $\delta^{-1} = 1000$.}
  \label{fig:bandgap_nondilute}
\end{center}
\end{figure}

\medskip

Finally, in order to verify our conclusion from Lemma \ref{lem-dilute}, namely that $\omega_1^\ast = \max_\alpha \omega_{M, \alpha} \approx \omega_M$ when $\alpha$ is non-zero, we fix the contrast to
 be $\delta^{-1} = 1000$ and observe $\omega_1^\ast$ and $\omega_M$ over a range of bubble sizes in Figure \ref{fig:dilute-regime}.
 
\begin{figure}
\begin{center}
\includegraphics[width = 6in]{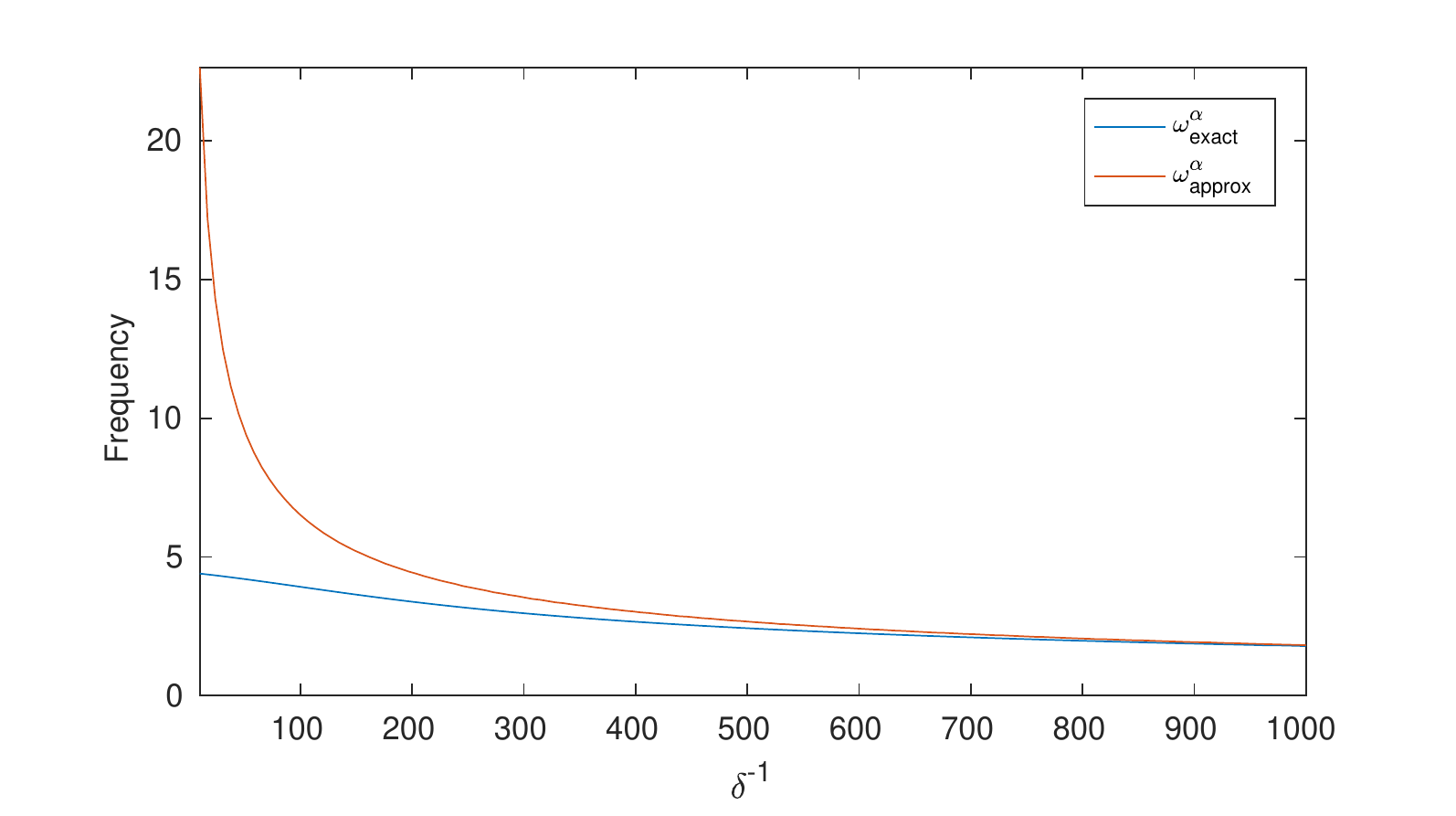}
  \caption{When the contrast $\delta^{-1}$ is sufficiently large, the $\alpha$-quasi-periodic resonance $\omega_1^\alpha$ given by Theorem \ref{approx_thm} provides a highly accurate approximation of the true resonance $\omega_\text{exact}^\alpha$.}
\label{fig:alpha-qp-formula}
\end{center}
\end{figure}

\begin{figure} [ht]
\begin{center}
\includegraphics[width = 6in]{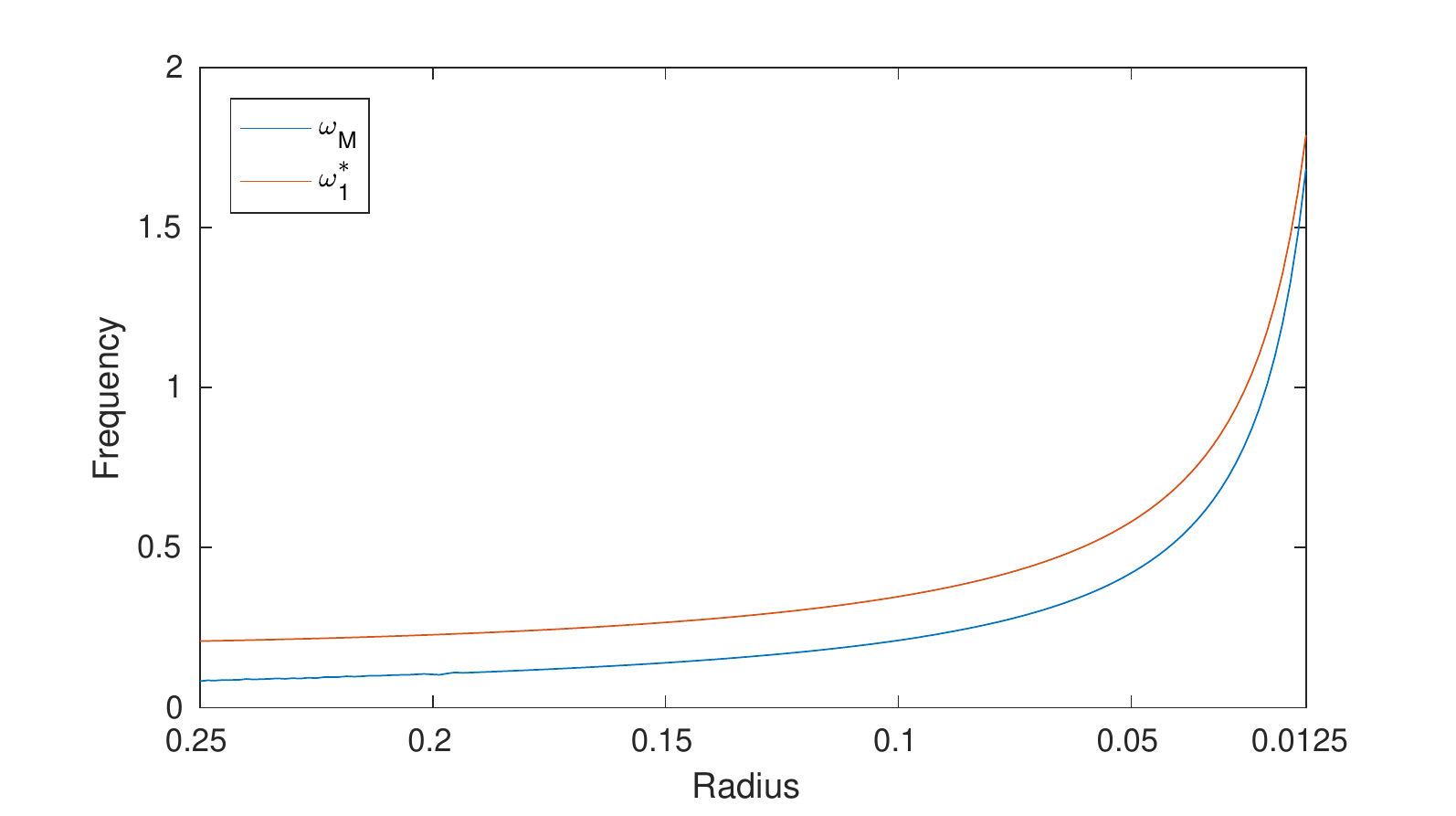}
  \caption{As the bubbles becomes smaller, the maximum frequency in the first band of the spectrum of the operator $\mathcal{A}(\omega, \delta)$, $\omega_1^\ast$,  approaches the Minnaert resonant frequency of a single bubble $\om_M$.}
\label{fig:dilute-regime}
\end{center}
\end{figure}

\section{Concluding remarks}
In this paper we have proved the existence of a subwavelength bandgap opening in bubble phononic crystals. We have illustrated our main findings with a variety of numerical experiments. We have also covered recently proved results on the effective medium theory in the dilute case. In a forthcoming work we will use the Bloch wave decomposition for homogenizing \cite{bloch, Lipton2} the bubble crystal near the maximum over $\alpha$ of $\omega_1^\alpha$ in the general case. Our aim is to prove that above such a frequency the crystal behaves like a material with a negative density while below it it behaves like a high contrast material, leading to superfocusing of acoustic waves.  

\appendix
\section{Some asymptotic expansions} \label{sec-appendix-1}

We recall some basic asymptotic expansions for the layer potentials in three dimensions from \cite{Ammari2009_book} (see also the appendix in \cite{Ammari2015_c}).

%%%%%%%%%%%%%%%%%%%%%%%%%%%%%%%%%%
\subsection{Asymptotic expansions of $\mathcal{S}_{D}^{k}$ and $\mathcal{K}_{D}^{k, *}$}

We expand the Green's function $G_k$ with
\begin{equation} \label{eq:defGk}
	G^k(\bx) = - \f{ \re^{ \ri k|\bx|}}{4 \pi | \bx |} = G_0(\bx) + \sum_{n=1}^\infty k^n G_n(\bx), \quad \text{with} \quad G_n(\bx) := - \frac{\ri^{n}}{4 \pi n!} | \bx |^{n-1}.
\end{equation}
In particular, $G_1(\bx) = -\frac{\ri}{4 \pi}$. Developing in power of $k$ the equation $(\Delta + k^2) G_k = \delta_0$ leads to
\begin{equation} \label{eq:magicId}
	\forall n \ge 1, \quad \Delta G_{n+2} = - G_n.
\end{equation}
From~\eqref{eq:defGk}, we decompose the single layer potential as
\begin{equation} \label{series-s}
\mathcal{S}_{D}^{k} =  \cS_D + \sum_{n=1}^\infty k^n \cS_{D,n} \quad \text{with} \quad \cS_{D,n}[\psi] := \int_{\p D} G_n(\bx - \by) \psi(\by) \rd \by,
\end{equation}
where the convergence holds in $\mathcal{B}(L^2(\p D), H^1(\p D))$. Similarly, the asymptotic expansion for the operator $\mathcal{K}_{D}^{k, *}$ is
\begin{equation} \label{series-k}
\mathcal{K}_{D}^{k, *} = \cK_D^* + \sum_{n=1}^\infty k^n \cK_{D,n}^* \quad \text{with} \quad \cK_{D,n}^*[\psi] := \int_{\p D} \dfrac{ \p G_n(\bx - \by)}{\p \nu_\bx} \psi(\by) \rd \by,
\end{equation}
where the convergence holds in $\mathcal{B}(L^2(\p D), L^2(\p D))$. Using~\eqref{eq:magicId}, we deduce the following useful identities.
\begin{lem} \label{lem:magicId_3d}
It holds: 
\begin{align*}
(i) & \quad \mathcal{K}_{D, 2}[\1_{\p D}] (x) = \int_{\p D}  \dfrac{ \p G_2(\bx - \by)}{\p \nu_\by} \rd \sigma(\by) = \int_{D}  \Delta_\by G_2(\bx - \by) \rd \by = - \int_D G_0(\bx - \by) \rd \by, \\
(ii) & \quad \mathcal{K}_{D, 3}[\1_{\p D}](x) = \int_{\p D}  \dfrac{ \p G_3(\bx - \by)}{\p \nu_\by} \rd \sigma(\by) = \int_{D}  \Delta_\by G_3(\bx - \by) \rd \by = - \int_D G_1(\bx - \by) \rd \by = \frac{\ri | D |}{4 \pi}.
\end{align*}
\end{lem}

\subsection{Asymptotic expansions of $\mathcal{S}_{D}^{\alpha,k}, ~(\mathcal{K}_{D}^{\alpha, k})^ *$}

For the $\alpha$-quasi-periodic Green's function $G^{\alpha,k}$, we have
\begin{align}
G^{\alpha,k}(\bx,\by)=G^{\alpha,0}+ \sum_{\ell=1}^\infty k^{2\ell} G_\ell^{\alpha,\#}:= G^{\alpha,0}(\bx,\by) - \sum_{\ell=1}^\infty k^{2\ell}\sum_{n\in\mathbb{Z}^3} \frac{e^{\ri (2\pi n + \alpha)\cdot (\bx-\by)}}{|2\pi n +\alpha|^{2(\ell+1)}}, \label{eq:defGk2}
\end{align}
when $\alpha \ne 0$, and $k \rightarrow 0$.

From~\eqref{eq:defGk2}, we decompose the single layer potential as
\begin{equation} \label{series-s2}
\mathcal{S}_{D}^{\alpha,k} =  \cS_D^{\alpha, 0} + \sum_{\ell=1}^\infty k^{2\ell }\cS_{D,\ell}^{\alpha} \quad \text{with} \quad \cS_{D,\ell}^\alpha[\psi] := \int_{\p D} G_\ell^{\alpha,\#}(\bx - \by) \psi(\by) \rd \by,
\end{equation}
where the convergence holds in $\mathcal{B}(L^2(\p D), H^1(\p D))$. Similarly, the asymptotic expansion for the operator $\left(\mathcal{K}_{D}^{-\alpha,k}\right)^*$ is
\begin{equation} \label{series-k2}
(\mathcal{K}_{D}^{-\alpha,k})^*   =( \cK_D^{-\alpha,0})^* + \sum_{\ell=1}^\infty k^{2\ell} (\cK_{D,\ell}^{\alpha})^* \quad \text{with} \quad  (\cK_{D,\ell}^\alpha)^*[\psi] (\bx):= \int_{\p D} \dfrac{ \p G_\ell^{\alpha,\#}(\bx - \by)}{\p \nu_\bx} \psi(\by) \rd \by,
\end{equation}
where the convergence holds in $\mathcal{B}(L^2(\p D), L^2(\p D))$.

\section{The two-dimensional case} \label{sec-appendix-2}

The aim of this appendix is to check that formula (\ref{o_1_alpha}) holds in the two-dimensional case, where $\omega_M$ is the (free space) Minnaert resonant frequency and $c_2$ is defined by (\ref{defc2}).
Note that for $\alpha\neq 0$,  the quasi-periodic single layer operator $\mathcal{S}_D^{\alpha,0} : L^2(\p D) \rightarrow H^1(\p D)$ is invertible. Moreover, the definitions (\ref{capacity})  and (\ref{capacityalpha})  of both the capacity and the $\alpha$-quasi-periodic capacity remain valid.

Using the asymptotic expansions in \cite[Appendix A]{H3a} as $k\rightarrow 0$, 
$$
\begin{array}{lll}
\mathcal{S}_{D}^{k}&=&  \hat{\mathcal{S}}_{D}^k +k^{2}\ln k \mathcal{S}_{D, 1}^{(1)}+k^{2} \mathcal{S}_{D, 1}^{(2)} + O(k^4 \ln k),\\
\mathcal{K}_{D}^{k,*}&=& \mathcal{K}_{D} +k^{2}\ln k \mathcal{K}_{D, 1}^{(1)}+k^{2} \mathcal{K}_{D, 1}^{(2)} + O(k^4 \ln k),
\end{array}
$$
where for $\psi\in L^2(\partial D)$
$$\begin{array}{lll}
\mathcal{S}_{D, j}^{(1)} [\psi](x) &=& \ds \int_{\p D} b_j|x-y|^{2j} \psi(y)d\sigma(y),\\
\mathcal{S}_{D, j}^{(2)} [\psi](x) &=& \ds \int_{\p D} |x-y|^{2j}(b_j\ln|x-y|+c_j)\psi(y)d\sigma(y),
\end{array}$$
and 
$$\begin{array}{lll}
\mathcal{K}_{D, j}^{(1)} [\psi](x) &=& \ds \int_{\p D} b_j\dfrac{\partial |x-y|^{2j}}{\partial \nu(x)}\psi(y)d\sigma(y),\\
\mathcal{K}_{D, j}^{(2)} [\psi](x) &=& \ds \int_{\p D} \dfrac{\partial \left( |x-y|^{2j}(b_j\ln|x-y|+c_j)\right)}{\nu(x)}\psi(y)d\sigma(y),
\end{array}$$
with 
$$
b_j = \dfrac{(-1)^j}{2\pi}\dfrac{1}{2^{2j}(j!)^2}, \quad c_j = b_j\left(\gamma -\ln 2-\dfrac{\ri \pi}{2}-\sum_{n=1}^j\dfrac{1}{n}\right),
$$
and
$\gamma$ being the Euler constant.

Therefore, in the two-dimensional case the asymptotic expansion (\ref{expdA}) should be replaced with 
\[
\mathcal{A}(\omega, \delta):=\mathcal{A}_0 + \mathcal{B}(\omega, \delta)
= \mathcal{A}_0 + \omega^2 \ln \omega \mathcal{A}_{1,1, 0}+ \omega^2 \mathcal{A}_{1, 2, 0}
+ \delta \mathcal{A}_{0, 1}+ O(\delta \omega^2 \ln \omega)  + O(\omega^4 \ln \omega),
\]
where $\mathcal{A}_0$ is defined by (\ref{defA0}), 
\[
\mathcal{A}_{1, 1,0}= 
\begin{pmatrix}
  v_b^2\mathcal{S}_{D,1}^{(1)} &  -v^2\mathcal{S}_{D,1}^{(1)}  \\
  v_b^2\mathcal{K}_{D,1}^{(1)}& 0
\end{pmatrix},
\,\, 
\mathcal{A}_{1, 2,0}= 
\begin{pmatrix}
  v_b^2\left( \ln v_b \mathcal{S}_{D,1}^{(1)} + \mathcal{S}_{D,1}^{(2)}\right) &  -  v^2\left( \ln v \mathcal{S}_{D,1}^{(1)} + \mathcal{S}_{D,1}^{(2)}\right)  \\
   v_b^2\left( \ln v_b \mathcal{K}_{D,1}^{(1)} + \mathcal{K}_{D,1}^{(2)}\right)& 0
\end{pmatrix},
\]
and 
\[
\mathcal{A}_{0, 1}=
\begin{pmatrix}
0& 0\\
0 &  -(\f{1}{2}I + \mathcal{K}_{D}^*)
\end{pmatrix}.
\]

Using the definition of the free space Minnaert resonance in dimension two in \cite[Theorem B1]{H3a}, it is not difficult to see that  (\ref{o_1_alpha}) holds.

\section{Multipole expansion method}\label{sec:appendix_multipole}

When $D$ is a circular disk of radius $R$, the integral equation admits an explicit representation. In this case, the solution can be represented as a sum of cylindrical waves $J_n(k r)e^{\ri n\theta}$ or $H_n^{(1)}(k r)e^{\ri n\theta}$. Here we give a multipole expansion interpretation of the integral operator $\mathcal{A}$, which leads to an efficient numerical scheme for computing its bandgap structure.

Recall that, for each fixed $k,\alpha$, we have to find a characteristic value of $\mathcal{A}(\omega, \delta)$ defined by
 \begin{equation} \label{aalpha2}
 \mathcal{A}(\omega, \delta)  = \left(\begin{array}{cc}
 \mathcal{S}^{k_b}_D  & - \mathcal{S}^{\alpha,
k}_D  \\ \nm \ds \frac{\p \mathcal{S}_D^{{k}_b}}{\p\nu}\Big|_- &
 \ds -\delta\frac{\p \mathcal{S}_D^{\alpha,k}}{\p\nu}\Big|_+ \end{array} \right).
 \end{equation}
%we replaced $\omega$ in the original $\mathcal{A}^{\alpha,k}$ by $\sqrt{k}\omega$. Then the corresponding solution represents TM mode solution and $k$ represents the permittivity of the inclusion. 
From the above expression, we see that $\mathcal{A}(\omega, \delta)$ is represented in terms of the single layer potential only.
So it is enough to derive a multipole expansion version of the single layer potential.

%Before computing $\mathcal{S}^{\alpha,\omega}[\varphi]_D$, 
Let us first consider the single layer potential $\mathcal{S}_D^k[\varphi]$ for a single disk $D$. We adopt the polar coordinates $(r,\theta)$. Then, since $D$ is a circular disk, the density function $\varphi=\varphi(\theta)$ is a $2\pi$-periodic function. It admits the following Fourier series expansion:
$$
\varphi = \sum_{n\in\mathbb{Z}} a_n e^{\ri n\theta},
$$
for some coefficients $a_n$.
Hence we only need to compute $u:=\mathcal{S}_D^{k}[e^{\ri n\theta}]$ which satisfies
\begin{equation} %\label{01p-100}
\left\{
\begin{array}{ll}
\ds  \Delta u + k^2 u=0 \quad  \mbox{in } \mathbb{R}^2\setminus
\overline{D}, \\
\nm
\ds \Delta u + k^2  u=0 \quad  \mbox{in } D, \\
\nm
\ds u|_{+}=u|_{-} \quad \mbox{on } \p D, \\
\nm
\ds  \pd{u}{\nu} \Big|_{+} - \pd{u}{\nu} \Big|_{-} = e^{\ri n\theta} \quad \mbox{on } \p D, \\
\nm u \mbox{ satisfies the Sommerfeld radiation condition}.
\end{array}
\right.
\end{equation}
The above equation can be easily solved by the separation of variables technique in polar coordinates. It gives
\begin{equation}
\label{SingleLayer_multipole}
\mathcal{S}_D^{k}[e^{\ri n\theta}] = 
\begin{cases}
\ds c J_n(k R) H_n^{(1)}(k r)e^{\ri n\theta}, &\quad |r|>R,
\\[0.5em]
\ds c H_n^{(1)}(k R) J_n(k r)e^{\ri n\theta}, &\quad |r|\leq R,
\end{cases}
\end{equation}
where $c=\frac{-\ri \pi R}{2}$.

Now we compute the quasi-periodic single layer potential $\mathcal{S}_D^{\alpha,k}[e^{\ri n\theta}]$. Since 
$$ %\label{eq:quasi-per-helmholtz-g}
G_\sharp^{\alpha,k} (x,y) = -
\frac{ \ri}{4} \sum_{m \in \mathbb{Z}^2} H^{(1)}_0(k |x - y - m|)
e^{ \ri m \cdot\alpha},
$$
we have
\begin{align*}
\mathcal{S}_D^{\alpha,k}[e^{\ri n\theta}]
&=\mathcal{S}_D^{k}[e^{\ri n\theta}]+\sum_{m\in\mathbb{Z}^2, m\neq 0} \mathcal{S}^{k}_{D+m}[e^{\ri n\theta}]e^{\ri m\cdot\alpha}
\\
&=\mathcal{S}_D^{k}[e^{\ri n\theta}]+c J_n(k R) \sum_{m\in\mathbb{Z}^2} H_n^{(1)}(k r_m)e^{\ri n\theta_m}e^{\ri m\cdot\alpha}.
\end{align*}
Here, $D+m$ means a translation of the disk $D$ by $m$ and $(r_m,\theta_m)$ is the polar coordinates with respect to the center of $D+m$.
By applying the following addition theorem:
$$
H_n^{(1)}(k r_m)e^{\ri n\theta_m} = \sum_{l\in \mathbb{Z}}(-1)^{n-l} H_{n-l}^{(1)}(k|m|)e^{\ri n\arg(m)} J_l(k r)e^{\ri l\theta}, 
$$
we obtain
\begin{equation}\label{QuasiSingleLayer_multipole}
\mathcal{S}_D^{\alpha,k}[e^{\ri n\theta}] = 
\mathcal{S}_D^{k}[e^{\ri n\theta}]+
cJ_n(k R) \sum_{l \in \mathbb{Z}} (-1)^{n-l}Q_{n-l} J_l(k r)e^{\ri l\theta}.
\end{equation}
where $Q_n$ is so called the lattice sum defined by
$$
Q_n := \sum_{m\in \mathbb{Z}^2, m\neq 0} H_n^{(1)}(k|m|)e^{\ri n \arg(m)}e^{\ri m\cdot \alpha}.
$$
So, from \eqnref{SingleLayer_multipole} and \eqnref{QuasiSingleLayer_multipole}, we finally obtain an explicit representation of $\mathcal{S}_D^{\alpha,k}$. 
For an efficient method for computing the lattice sum $Q_n$, see \cite{Linton2010}.

In the numerical computations, we should consider the truncated series 
$$\sum_{n=-N}^{N}a_n\mathcal{S}_D^{\alpha,\omega}[e^{\ri n\theta}]$$ instead of 
$\mathcal{S}_D^{\alpha,k}[\varphi]=\sum_{n\in \mathbb{Z}}a_n\mathcal{S}_D^{\alpha,k}[e^{\ri n\theta}]$ for some large enough $N\in \mathbb{N}$. Then,  using $e^{\ri n\theta}$ as basis, we have the following matrix representation of the operator $\mathcal{S}^{\alpha,k}$: 
$$
\mathcal{S}_D^{\alpha,k}[\varphi]|_{\p D} \approx
\begin{pmatrix}
S_{-N,-N} & S_{-N,-(N-1)} & \cdots & S_{-N,N} \\
S_{-(N-1),-N} & S_{-(N-1),-(N-1)} & \cdots & S_{-(N-1),N} \\
\vdots    &              & \ddots & \vdots \\
S_{N,-N}     &  \cdots & \cdots & S_{NN}
\end{pmatrix}
\begin{pmatrix}
a_{-N}
\\
a_{-(N-1)}
\\
\vdots
\\
a_{N}
\end{pmatrix},
$$
where $S_{m,n}$ is given by
$$
S_{m,n} =  c J_n (k R)H_n^{(1)}(k R) \delta_{mn} + c J_n(k R)(-1)^{n-m} Q_{n-m}J_m(k R).
$$
Similarly, we also have the following matrix representation for $\frac{\p \mathcal{S}_D^{\alpha,k}}{\p\nu}|_{\p D}^\pm$:
$$
\frac{\p\mathcal{S}_D^{\alpha,k}}{\p\nu}[\varphi]\Big|^\pm_{\p D} \approx
\begin{pmatrix}
S'^\pm_{-N,-N} & S'^\pm_{-N,-(N-1)} & \cdots & S'^\pm_{-N,N} \\
S'^\pm_{-(N-1),-N} & S'^\pm_{-(N-1),-(N-1)} & \cdots & S'^\pm_{-(N-1),N} \\
\vdots    &              & \ddots & \vdots \\
S'^\pm_{N,-N}     &  \cdots & \cdots & S'^\pm_{NN}
\end{pmatrix}
\begin{pmatrix}
a_{-N}
\\
a_{-(N-1)}
\\
\vdots
\\
a_{N}
\end{pmatrix},
$$
where $S'^\pm_{m,n}$ is given by
\begin{align*}
S'^\pm_{m,n} &=   \pm \frac{1}{2}+ k c\Big(J_n \cdot (H_n^{(1)})'+J_n'\cdot H_n^{(1)}\Big)(k R) \delta_{mn}
\\
&\quad + c J_n(k R)(-1)^{n-m} Q_{n-m}k J_m'(k R).
\end{align*}
The matrix representation of $\mathcal{A}(\omega,\delta)$ immediately follows.

%%%%%%%%%%%%%%%%%%%%%%%%%%%%%%%%%%%%%%%%%%%%%%%%%%%%%%%%%%%%

\bibliography{bubble_phononic2}
\bibliographystyle{plain}

\end{document}